\documentclass[10pt,reqno]{amsart}
\usepackage{lineno,hyperref}
\usepackage{latexsym}
\usepackage{graphicx}
\usepackage{amsmath}
\usepackage{amsthm}
\usepackage{amssymb,amsfonts,enumerate}
\allowdisplaybreaks
\newtheorem{theorem}{Theorem}

\newtheorem{lemma}{Lemma}

\newtheorem{corollary}{Corollary}
\newtheorem{definition}{Definition}
\newtheorem{remark}{Remark}
\newtheorem{example}{Example}

\newcommand{\ba}{\begin{array}}
\newcommand{\ea}{\end{array}}
\newcommand{\ei}{\end{itemize}}
\newcommand{\bc}{\begin{center}}
\newcommand{\ec}{\end{center}}
\newcommand{\bfr}{\begin{flushright}}
\newcommand{\efr}{\end{flushright}}

\newcommand{\be}{\begin{equation}}
\newcommand{\ee}{\end{equation}}
\newcommand{\bea}{\begin{eqnarray}}
\newcommand{\eea}{\end{eqnarray}}
\newcommand{\bee}{\begin{eqnarray*}}
\newcommand{\eee}{\end{eqnarray*}}

\newcommand{\nii}{\noindent}

\begin{document}

\title[Fine spectra of generalized difference operator $\Delta_{a,b}$ on $\ell_p \ (1<p<\infty)$]{On some study of the fine spectra of generalized difference operator $\Delta_{a,b}$ on $\ell_p \ (1<p<\infty)$}

%    Remove any unused author tags.

\author{Riddhick Birbonshi}
\address{Riddhick Birbonshi, Department of mathematics, Indian Institute of Technology Kharagpur, India 721302}
\curraddr{}
\email{riddhick.math@gmail.com}
\thanks{}%    author one information

\author{Arnab Patra}
\address{Arnab Patra, Department of mathematics, Indian Institute of Technology Kharagpur, India 721302}
\curraddr{}
\email{arnptr91@gmail.com}
\thanks{}

%    author two information
\author{P. D. Srivastava}
\address{P. D. Srivastava, Department of mathematics, Indian Institute of Technology Kharagpur, India 721302}
\curraddr{}
\email{pds@maths.iitkgp.ernet.in}
\thanks{}

\subjclass[2010]{Primary: 47A10; 47B37}

\keywords{Spectrum of an operator; Infinite matrices; Sequence spaces.}

\date{}

\dedicatory{}

%\begin{frontmatter}
%
%\title{On some study of the fine spectra of generalized difference operator $\Delta_{a,b}$ on $\ell_p \ (1<p<\infty)$}
%%\tnotetext[mytitlenote]{Fully documented templates are available in the elsarticle package on \href{http://www.ctan.org/tex-archive/macros/latex/contrib/elsarticle}{CTAN}.}
%
%\author[]{Riddhick Birbonshi\corref{cor1}}
%\ead{riddhick.math@gmail.com}
%\author[]{Arnab Patra}
%\ead{arnptr91@gmail.com}
%
%
%\author[]{P. D. Srivastava}
%\ead{pds@maths.iitkgp.ernet.in}
%
%
%\address{Department of Mathematics\\ Indian Institute of Technology Kharagpur,
%Kharagpur - 721302, India}
%
%
\begin{abstract}
In this paper, we determine the spectrum, the point spectrum, the continuous spectrum and the residual spectrum of the generalized  difference operator $\Delta_{a,b}$ on the sequence space $\ell_p \ (1< p < \infty)$ where the real sequences $a=\{a_k\}$ and $b=\{b_k\}$ are not necessarily convergent. Hence our results generalize the work given by Akhmedov and El-Shabrawy [Math. Slovaca 65~(5) (2015) 1137--1152] for the sequence space $\ell_p (1< p <\infty)$.
\end{abstract}
%
%\begin{keyword}
% Spectrum of an operator; Infinite matrices; Sequence spaces
%\MSC[2010] Primary: 47A10; 47B37
%\end{keyword}
%
%\end{frontmatter}
%
%\linenumbers
\maketitle
\section{Introduction}
Several authors have studied the spectrum and fine spectrum of various bounded linear operators defined on various sequence spaces. Altay and Ba{\c{s}}ar \cite{Basar04} and Kayaduman and Furkan \cite{Kayaduman} have obtained the spctrum and fine spectrum of the difference operator $\Delta$ over the sequence spaces $c_0,$ $c$ and $l_1,$ $bv$ respectively. Also the fine spectrum of  $\Delta$ over the sequence space $\ell_p,  bv_p (1 \leq p < \infty)$ has been studied by Akhmedov and Ba{\c{s}}ar \cite{AAFB,AAFB1}.
The fine spectrum of the generalized difference operator $B(r, s)$ has been studied by Altay and Ba{\c{s}}ar \cite{brs_c0_c}, Furkan et al. \cite{brs_l1_bv} over the sequence spaces $c_0$ and $c$, $l_1$ and $bv$ respectively. While the fine spectrum of the operator $B(r, s)$ over the sequence spaces $l_p$ and $bv_p$ $(1<p< \infty)$ has been studied by Bilgi{\c{c}} and Furkan \cite{brs_lp_bvp}.
Furkan et al. \cite{brst_c0_c, FURK2010}, Bilgi{\c{c}} and Furkan \cite{brst_l1_bv} have further generalized these results for the operator $B(r, s, t).$ Later, Altun \cite{altun} has studied the fine spectrum of triangular Toeplitz operator over the sequence spaces $c_0$ and $c.$

\nii The fine spectrum of the operator $\Delta_v$ over the sequence spaces $c_0$ and $\ell_1$ have been studied by Srivastava and Kumar \cite{SK1, SK}. Some of their results have been revised by Akhmedov and El-Shabrawy \cite{AE2012}. The fine spectrum of the operator $\Delta_v$ over the sequence $c$ and $\ell_p \ (1<p<\infty)$ have been also studied by Akhmedov and El-Shabrawy \cite{AE2011a}. They have also modified the operator $\Delta_v$ and studied its fine spectrum on the space $c$ and $\ell_p \ (1<p<\infty)$. Later, Akhmedov and El-Shabrawy \cite{AE12} have modified the operator $\Delta_v$ and studied its fine spectrum on the spaces $c_0$ and $l_1$. Next the operator $\Delta_v$ is further generalized to the operator $\Delta_{a,b}$, where the sequences $\{a_k\}$ $\{b_k\}$ are two convergent sequences of real numbers. Several authors have studied the fine spectrum of the operator $\Delta_{a,b}$ over some sequence spaces by imposing more conditions on the sequences $\{a_k\}$ and $\{b_k\}$. The fine spectrum of the operator $\Delta_{a,b}$ over the sequence space $c_0$ has been studied by Akhmedov and El-Shabrawy \cite{AE2010}, while the fine spectrum of the operator $\Delta_{a,b}$ over the sequence spaces $c$ and $\ell_p (1<p<\infty)$ has been studied by Akhmedov and El-Shabrawy \cite{AE2011} and El-Shabrawy \cite{EL2012} respectively. The fine spectrum of the operator $\Delta_{u,v}$ on the sequence space $\ell_1$ has been also studied by Srivastava and Kumar \cite{SK2}. Recently El-Shabrawy \cite{EL14} and Akhmedov and El-Shabrawy \cite{AE15} have modified the operator $\Delta_{a,b}$ and studied its fine spectrum on the sequence spaces $c_0$ and $\ell_p (1\leq p <\infty)$ respectively. Recently Das \cite{das} has studied the spectrum and fine spectrum of a new difference operator $U \left(r_1, r_2; s_1, s_2 \right)$ over $c_0$. \\
%Hence this modified operator $\Delta_{a,b}$ gives the previous results given by \cite{AE2010} and \cite{EL2012} as particular case.\\
\nii In this paper, we consider the operator $\Delta_{a,b}$ defined on the sequence space $\ell_p \ (1<p<\infty)$, where the real sequences $a=\{a_n\}$ and $b=\{b_n\}$ are not necessarily convergent. The matrix form of the operator $\Delta_{a,b}$ is given as follows :
\begin{equation} \label{main}
   \Delta_{a,b}=
  \begin{bmatrix}
   a_1   & 0      & 0      & 0      & \cdots \\
   b_1   & a_2    & 0      & 0      & \cdots \\
   0     & b_2    & a_3    & 0      & \cdots \\
   0     & 0      & b_3    & a_4    & \cdots \\
  \vdots & \vdots & \vdots & \vdots & \vdots
  \end{bmatrix},
\end{equation}
where the real sequences $a=\{a_k\}$ and $b=\{b_k\}$ are either periodic sequences of period $m \ (m\geq1)$ or $a_{km-i}\to p_{m-i}$ and $b_{km-i}\to q_{m-i}$ as $k\to\infty$, for $i=0,1,2...m-1$ with
$b_k\neq0$ for all $k\in\mathbb{N}$, $q_j\neq0$ for $j\in\{1,2,...m\}$. Thus we have generalized the results given by Akhmedov and El-Shabrawy \cite{AE15} for the sequence space $\ell_p (1< p <\infty)$.

\section{Preliminaries and notation}
Now we give some results related to the spectrum of the operator $\Delta_{a,b}$ on the space $\ell_p \ (1< p < \infty)$. Let $X$ and $Y$ be two Banach spaces and $A : X\rightarrow Y$ be a bounded linear operator. We denote the range of $A$ as $R(A)$, that is, $R(A)=\{y\in Y: y=Ax, x\in X\}$. The set of all bounded linear operators on $X$ into itself is denoted by $B(X)$. Further, the adjoint of $A$, denoted by $A^*$, is a bounded linear operator on the dual space $X^*$ of $X$ defined by
\[
(A^*\phi)(x)=\phi(Ax) \mathrm{~for~all~} \phi\in X^* \mathrm{~and~} x\in X.
\]

Let $X\neq\{0\}$ be a complex normed linear space and $A: D(A)\rightarrow X$ be a linear operator with domain $D(A)\subseteq X$. With $A$, we associate the operator $A_\alpha=(A-\alpha I)$, where $\alpha$ is a complex number and $I$ is the identity operator on $D(A)$. The inverse of $A_\alpha$ (if exists) is denoted by $A_\alpha^{-1}$, where $A_\alpha^{-1}=(A-\alpha I)^{-1}$ and is known as the resolvent operator of $A$. It is easy to verify that $A_\alpha^{-1}$ is linear, if $A_\alpha$ is linear. The definitions and known results as given below will be used in the sequel.

\begin{definition}
Let $X\neq\{0\}$ be a complex normed linear space and $A:D(A)\rightarrow X$ be a linear operator with domain $D(A)\subseteq X$. A regular value of $A$ is a complex number $\alpha$ such that the following conditions R1-R3 hold $:$\\
(R1) $A_\alpha^{-1}$ exists\\
(R2) $A_\alpha^{-1}$ is bounded\\
(R3) $A_\alpha^{-1}$ is defined on a set which is dense in $X$.
\end{definition}

Resolvent set $\rho(A,X)$ of $A$ is the set of all regular values $\alpha$ of $A$. Its complement in the complex plane $\mathbb{C},$ i.e., $\sigma(A,X)=\mathbb{C} \setminus\rho (A,X)$ is called  the spectrum of $A$. The spectrum $\sigma(A,X)$ is further partitioned into three disjoint sets, namely, point spectrum, continuous spectrum and residual spectrum which are defined as follows:

 The point spectrum $\sigma_p(A,X)$ is the set of all $\alpha\in \mathbb{C}$ such that $A_\alpha^{-1}$ does not exist, i.e., the condition (R1) fails. The elements of $\sigma_p(A,X)$ are called eigenvalues of $A$.

 The continuous spectrum $\sigma_c(A,X)$ is the set of all $\alpha\in \mathbb{C}$ such that the conditions (R1) and (R3) hold but the condition (R2) does not hold, i.e., $A_\alpha^{-1}$ exists, domain of $A_\alpha^{-1}$ is dense in $X$ but $A_\alpha^{-1}$ is unbounded.

 The residual spectrum  $\sigma_r(A,X)$ is the set of all $\alpha\in \mathbb{C}$ such that $A_\alpha^{-1}$ exists but does not satisfy the condition (R3), i.e., domain of $A_\alpha^{-1}$ is not dense in $X$. In this case, the condition (R2) may or may not hold good.

Let $A \in B(X)$ and $A_{\lambda} = (A-\lambda I)$ for a complex number $\lambda$. Then Goldberg \cite{Q} considered three possibilities for $R(A_{\lambda})$ and $A_{\lambda}^{-1}:$
\begin{enumerate}
\item[(A)] $R(A_{\lambda}) = X,$
\item[(B)] $R(A_{\lambda})\neq \overline{R(A_{\lambda})} = X,$
\item[(C)] $\overline{R(A_{\lambda})} \neq X$
\end{enumerate}
and
\begin{enumerate}
\item[(1)] $A_{\lambda}$ is injective and $A_{\lambda}^{-1}$ is continuous,
\item[(2)] $A_{\lambda}$ is injective and $A_{\lambda}^{-1}$ is discontinuous,
\item[(3)] $A_{\lambda}$ is not injective
\end{enumerate}
If we combine the possibilities A, B, C and 1, 2, 3 then nine different states are created. These are labelled by $A_1$, $A_2$, $A_3$, $B_1$, $B_2$, $B_3,$ $C_1$, $C_2$, $C_3.$ For example, If the operator $A$ is in state $C_2$ for example, then $\overline{R(A_{\lambda})} \neq X$, $A_{\lambda}$ is injective and $A_{\lambda}^{-1}$ is discontinuous.
\begin{remark}\label{R1.7.7}
\rm If $\alpha$ is a complex number such that $A_\alpha\in A_1$ or
$A_\alpha\in B_1$, then $\alpha$ belongs to the resolvent set
$\rho(A,X)$ of $A$ on $X$. The further classification gives rise to
the fine spectrum of $A$. It is clear that,
\begin{eqnarray*}
\sigma_p(A,X)&=& A_3\sigma(A,X)\cup B_3\sigma(A,X)\cup C_3\sigma(A,X),\\
\sigma_c(A,X)&=& B_2\sigma(A,X)\\
{\rm and}~\sigma_r(A,X)&=& C_1\sigma(A,X)\cup C_2\sigma(A,X).
\end{eqnarray*}
\end{remark}

\noindent Let $M=(a_{nk})$ be an infinite matrix of complex numbers and $\lambda$ and $\mu$ be two sequence spaces. Then, the matrix $M$ defines a matrix mapping from $\lambda$ into $\mu$ if for every sequence $x=(x_k) \in \lambda$ the sequence $Mx = \{(Mx)_n\}$ is in $\mu$ where
\[(Mx)_n = \sum\limits_k a_{nk}x_k, \ n \in \mathbb{N} \]
and it is denoted by $M: \lambda \rightarrow \mu.$ We denote the class of all matrices $M$ such that $M: \lambda \rightarrow \mu$ by $(\lambda, \mu).$

\begin{lemma} \cite[p. 59]{Q} \label{denserange}
The bounded linear operator $A : X \rightarrow Y$ has dense range if and only if $A^*$ is one to one
\end{lemma}

\begin{lemma} \cite[p. 126]{wilansky} \label{boundedl1}
A matrix $A=(a_{nk})$ gives rise to a bounded linear operator $A \in B(l_1)$ from $l_1$ to itself if and only if the supremum of $l_1$ norms of the columns of A is bounded.
\end{lemma}

\begin{lemma} \cite[p. 126]{wilansky} \label{boundedl}
A matrix $A=(a_{nk})$ gives rise to a bounded linear operator $A \in B(l_\infty)$ from $l_\infty$ to itself if and only if the supremum of $l_1$ norms of the rows of A is bounded.
\end{lemma}

\begin{lemma} \cite[p. 174, Theorem 9]{maddox} \label{boundedlp}
Let $1 < p < \infty$ and suppose $A \in (l_1, l_1) \cap (l_\infty, l_\infty)$. Then $A \in  (l_p, l_p).$
\end{lemma}

Before proceeding, we list some notations which will be used throughout this paper.\\

\noindent \textbf{Notations :}
\begin{eqnarray*}
S_1 &=&  \Big\{ \lambda \in \mathbb{C}: \ |\lambda - p_1| \ |\lambda - p_2| \ \cdots | \ \lambda - p_m| \leq |q_1 q_2 \cdots q_m| \Big\},\\
S_2 &=& \Big\{ a_k : \ a_k \notin S_1 \ \mbox{ for } \ k \in \mathbb{N}\Big\},\\
S_3 &=& \Bigg\{ a_k :|a_k - p_1| \ \cdots \ |a_k - p_m| = |q_1 \cdots q_m| \ \mbox{ and} \\ && \sum\limits_{i=s}^{\infty} \Big|\frac{b_s b_{s+1} \cdots b_i}{(a_{s+1} - a_k)\cdots (a_{i+1} - a_k)} \Big|^p < \infty \ \mbox{ for some } \ (k\leq) \ s \in \mathbb{N} \Bigg\},\\
S_4 &=&  \Big\{\lambda \in \mathbb{C}: \ |\lambda - p_1| \ |\lambda - p_2| \ \cdots \ |\lambda - p_m| < |q_1 q_2 \cdots q_m| \Big\},\\
S_5 &=& \Big\{\lambda \in \mathbb{C}: \ |\lambda - p_1| \ |\lambda - p_2| \ \cdots \ |\lambda - p_m| = |q_1 q_2 \cdots q_m| \Big\},\\
S_6 &=& \Bigg\{ \lambda \in \mathbb{C} : \ |\lambda - p_1| \cdots |\lambda - p_m| = |q_1 \cdots q_m| \ \mbox{ and} \\ && \sum\limits_{k=2} ^{\infty} \Big|\frac{(\lambda - a_1)(\lambda - a_2) \cdots (\lambda - a_{k-1})}{b_1 b_2 \cdots b_{k-1}} \Big|^q <\infty \Bigg\}.
%S_7 &=& \Bigg\{a_k : \ |a_k - p_1| \ \cdots \ |a_k - p_m| = |q_1 \cdots q_m| \ \mbox{ and} \\ && \sum\limits_{i=s}^{\infty} \Big|\frac{b_s b_{s+1} \cdots b_i}{(a_{s+1} - a_k)\cdots (a_{i+1} - a_k)} \Big|^p  \ \mbox{ is divergent for all } \ (k\leq) \ s  \in \mathbb{N}  \Bigg\}\\
%S_8 &=& \Bigg\{ \lambda : \ \lambda \neq a_k, \ \ |\lambda - p_1| \ \cdots \ |\lambda - p_m| = |q_1 \cdots q_m| \ \mbox{ and} \\ && \sum\limits_{k=2} ^{\infty} \Big|\frac{(\lambda - a_1)(\lambda - a_2) \cdots (\lambda - a_{k-1})}{b_1 b_2 \cdots b_{k-1}} \Big|^q  \ \mbox{ is divergent}  \Bigg\}.
\end{eqnarray*}
where $q$ is such that $\frac{1}{p}+\frac{1}{q}=1$.

\section{Spectrum and fine spectrum of the operator $\Delta_{a,b}$ on the sequence space $\ell_p$}
In this section, we determine the spectrum and fine spectrum of the generalized difference operator $\Delta_{a,b}$ over the sequence space $\ell_p$, where $1<p<\infty$. Throughout this paper, we assume that the real sequences $a=\{a_k\}$ and $b=\{b_k\}$ \begin{enumerate}[(i)]
\item either periodic sequences of period $m \ (m\geq1)$  i.e. either $a=\{a_k\}$ is of the form $\{a_1,a_2,\ldots,a_m, \ a_1,a_2,\ldots,a_m,\ldots\}$ \\
and $b=\{b_k\}$ is of the form $\{b_1,b_2,\ldots,b_m, \ b_1,b_2,\ldots,b_m,\ldots\}$ with $b_k\neq 0$ for all $k\in \{1,2,...,m\}$ or
\item $a_{km-i}\to p_{m-i}$ and $b_{km-i}\to q_{m-i}$ as $k\to\infty$, for $i=0,1,2...m-1$. Here $b_k\neq0$ for all $k\in\mathbb{N}$, $q_j\neq0$ for $j\in\{1,2,...m\}$.
\end{enumerate}
{\bf Note :} When $\{a_k\}$ and $\{b_k\}$ are periodic with period $m \ (\geq 1)$, in that case $p_i=a_i$ and $q_i=b_i$ for all $i\in\{1,2,\cdots,m\}$.\\

\begin{theorem}
$\Delta_{a,b}:\ell_p\to \ell_p$ is a bounded linear operator.
\end{theorem}
\begin{proof}
The linearity of $\Delta_{a,b}$ is trivial. To show the boundedness, let $\{x_k\}\in \ell_p$. Then by Minkowski's inequality, we have
\bee
\|\Delta_{a,b} (x)\|_{\ell_p} &=& \Big(\sum_{k=1} ^{\infty} |a_kx_k + b_{k-1}x_{k-1}|^p \Big)^{1/p}\\
&\leq& \Big( \sum_{k=1} ^{\infty} |a_kx_k|^p\Big)^{1/p}+ \Big( \sum_{k=2} ^{\infty} |b_{k-1}x_{k-1}|^p\Big)^{1/p}\\
&\leq& \sup\limits_k |a_k| \Big( \sum_{k=1} ^{\infty} |x_k|^p\Big)^{1/p} + \sup\limits_k |b_k| \Big( \sum_{k=1} ^{\infty} |x_k|^p\Big)^{1/p}\\
&=& \sup\limits_k \big(|a_k|+|b_k|\big) \ \|x\|_{\ell_p} \ < \infty
\eee
Therefore $\Delta_{a,b}:\ell_p\to \ell_p$ is a bounded linear operator.
\end{proof}

\noindent Before going to the results on the spectrum of $\Delta_{a,b}$ we first prove a lemma which is useful.
\begin{lemma}\label{lm3.1}
Let $\lambda\in\mathbb{C}$ be such that
\begin{enumerate}[(I)]
\item $\lambda\neq a_k \ $ for all $k\in\mathbb{N}$ and
\item $|\lambda - p_1| \ |\lambda - p_2| \ \cdots \ |\lambda - p_m| > |q_1 q_2 \cdots q_m|$, where $p_i,q_i\in\mathbb{R}$ and $q_i\neq 0$ for $i=\{1,2,\ldots, m\}$
\end{enumerate}
Then there exist $\alpha>0$ and $N_0 \in \mathbb{N}$ such that for $k \geq N_0$, the followings hold:
\begin{enumerate}
\item[(i)]{\fontsize{10}{10}\begin{eqnarray*}
\frac{1}{|a_k - \lambda |}\leq \alpha, \ \ \left|\frac{b_k}{(a_k - \lambda)(a_{k+1} - \lambda)} \right| \leq \alpha,  \ \cdots \  , \left|  \frac{b_k b_{k+1}\cdots b_{k+m-2}}{(a_{k} - \lambda)(a_{k+1} - \lambda)\cdots (a_{k+m-1} - \lambda)} \right|\leq \alpha,
\end{eqnarray*}}
%\begin{eqnarray*}
%&& \frac{1}{|a_k - \lambda |} \leq \alpha, \\
%&& \left|\frac{b_k}{(a_k - \lambda)(a_{k+1} - \lambda)} \right| \leq \alpha, \\
%&& \hspace{2cm} \vdots \\
%&& \left|  \frac{b_k b_{k+1}\cdots b_{k+m-2}}{(a_{k} - \lambda)(a_{k+1} - \lambda)\cdots (a_{k+m-1} - \lambda)} \right| \leq \alpha.\\
%\end{eqnarray*}
\item[(ii)]{\fontsize{10}{10}\begin{eqnarray*}
\frac{1}{|a_{k+m-1} - \lambda |} \leq \alpha, \ \ \left|\frac{b_{k+m-2}}{(a_{k+m-1} - \lambda)(a_{k+m-2} - \lambda)} \right| \leq \alpha,  \ \cdots  \ , \left|  \frac{b_{k+m-2} \cdots b_{k+1}b_k}{(a_{k+m-1} - \lambda)(a_{k+m-2} - \lambda)\cdots (a_k-\lambda)} \right| \leq \alpha.
\end{eqnarray*}}
%\begin{eqnarray*}
% && \frac{1}{|a_{k+m-1} - \lambda |} \leq \alpha, \\
%&& \left|\frac{b_{k+m-2}}{(a_{k+m-1} - \lambda)(a_{k+m-2} - \lambda)} \right| \leq \alpha, \\
%&& \hspace{2cm} \vdots \\
%&& \left|  \frac{b_{k+m-2} \cdots b_{k+1}b_k}{(a_{k+m-1} - \lambda)(a_{k+m-2} - \lambda)\cdots (a_{k+1} - \lambda)(a_k-\lambda)} \right| \leq \alpha.
%\end{eqnarray*}
\end{enumerate}
\end{lemma}
\begin{proof}
Let
$$f(\lambda) = \frac{(\lambda - p_1)(\lambda - p_2) \cdots (\lambda - p_m)}{q_1 q_2 \cdots q_m}$$ and $\mathbb{D} = \{z \in \mathbb{C}: |z|<1\}$. Clearly $f: \mathbb{C} \rightarrow \mathbb{C}$ is a continuous function so, $S_1 ^0 = f^{-1}(\mathbb{D})$ is an open set with $p_1, p_2, \ \cdots \ , p_m$ are its some interior points. So there exists $m$ open balls contained in $S_1 ^0$ with center $p_1, p_2, \ \cdots \ ,p_m$ and suitable radius $\varepsilon$. Clearly $\varepsilon <\min\limits_{i} \mbox{dist}(p_i, \partial S_1)$. Now if $\lambda$ satisfies the condition (II), then for $\varepsilon>0$ there exists $N_0 \in \mathbb{N}$ and a number $i \in \{1,2, \cdots, m \}$ such that for $k\geq N_0$, we have
\begin{eqnarray*}
|a_k - \lambda| &=& |a_k - p_i + p_i - \lambda|\\
 &\geq& |p_i - \lambda| - |a_k - p_i| \\
 &\geq& \mbox{dist}(p_i, \partial S_1) - \varepsilon \ \ \mbox{ as $\lambda\notin S_1$}\\
\mbox{Similarly, } \ |a_{k+j} - \lambda| &\geq & \mbox{dist}(p_{i+j}, \partial S_1) - \varepsilon, \ \ \mbox{for} \ \ j = 1, 2, \cdots , m-1.\\
\mbox{where }
p_{i+j} &=& \begin{cases}
p_{i+j},& \mbox{ if } {i+j}\leq m\\
p_{(i+j)-m},& \mbox{ if } {i+j}>m
\end{cases}
\end{eqnarray*}
%as $\lim\limits_{k\to\infty} a_{km-i}=p_{m-i}$ and $i$ varies from 1 to $m$.\\
Now if we choose $r_{0} = \min\limits_{i} \mbox{dist}(p_i, \partial S_1)$ and $r_{0} - \varepsilon = r_{1}$ then $r_1>0$. So, for $k\geq N_0$, we have $|a_{k+j} - \lambda| \geq r_{1}$ for $j = 0, 1, \cdots , m-1$. Since the sequence $\{b_k\}$ is bounded, so $|b_k| \leq t$ for all $k\in\mathbb{N}$. Hence there exist a real number $\alpha>0$ and $N_0 \in \mathbb{N}$ such that for $k\geq N_0$,
{\fontsize{10}{10}\begin{eqnarray*}
\frac{1}{|a_k - \lambda |}\leq \alpha, \ \ \left|\frac{b_k}{(a_k - \lambda)(a_{k+1} - \lambda)} \right| \leq \alpha, \ \ \cdots \ \ , \left|  \frac{b_k b_{k+1}\cdots b_{k+m-2}}{(a_{k} - \lambda)(a_{k+1} - \lambda)\cdots (a_{k+m-1} - \lambda)} \right|\leq \alpha,
\end{eqnarray*}}
where $\alpha=\max\big\{\frac{1}{r_1},\frac{t}{r_1 ^2},\ldots, \frac{t^{m-1}}{r_1 ^m}\big\}$.\\
Also (ii) can be shown in a similar way. This completes the proof.
\end{proof}

\begin{theorem}\label{sepct_p}
The spectrum of $\Delta_{a,b}$ on $\ell_p$ is given by
\begin{equation*}
\sigma(\Delta_{a,b}, l_p) = S_1 \cup S_2.
\end{equation*}
%\begin{eqnarray*}
%\mbox{where} \ \ S_1 &=&  \{\lambda \in \mathbb{C}: |\lambda - p_1||\lambda - p_2| \cdots |\lambda - p_m| \leq |q_1 q_2 \cdots q_m| \},\\
%S_2 &=& \{a_k : k \in \mathbb{N} \ \mbox{and} \ a_k \notin S_1\}.
%\end{eqnarray*}
\end{theorem}
\begin{proof}
First we prove that $\sigma(\Delta_{a,b}, l_p)\subseteq S_1\cup S_2$. For this, we show that if $\lambda\notin S_1\cup S_2$, then $(\Delta_{a,b}-\lambda I)^{-1}$ exists and is in $B(\ell_p)$.\\
Let $\lambda \notin S_1 \cup S_2$. Then $$|\lambda - p_1||\lambda - p_2| \cdots |\lambda - p_m| > |q_1 q_2 \cdots q_m|$$ and $\lambda \neq a_k$ for all $k \in \mathbb{N}.$ So, $(\Delta_{a,b}-\lambda I)$ is triangle and hence $(\Delta_{a,b}-\lambda I)^{-1}$ exists.\\
Let $y=\{y_k\}\in\ell_p$. Then solving the equation
\begin{equation*}
\Big(\Delta_{a,b}-\lambda I\Big) \ x = y,
\end{equation*}
for $x=\{x_k\}$ in terms of $y$, we get
\begin{equation*}
x_k=\frac{(-1)^{k-1}b_1b_2\cdots b_{k-1}}{(a_1 - \lambda)(a_2 - \lambda)\cdots (a_k - \lambda)} \ y_1 +\cdots - \frac{b_{k-1}}{(a_{k-1} - \lambda)(a_k - \lambda)} \ y_{k-1} + \frac{1}{(a_k - \lambda)} \ y_k
\end{equation*}
for $k\in\mathbb{N}$. Now
\begin{equation}\label{mainop}
   (\Delta_{a,b}-\lambda I)^{-1}=
  \left( {\begin{array}{cccccc}
   \frac{1}{(a_1 - \lambda)}   & 0      & 0      & 0  & \cdots \\\\
   \frac{-b_1}{(a_1 - \lambda)(a_2 - \lambda)}   & \frac{1}{(a_2 - \lambda)}    & 0 & 0 & \cdots \\\\
   \frac{b_1b_2}{(a_1 - \lambda)(a_2 - \lambda)(a_3 - \lambda)}& \frac{-b_2}{(a_2 - \lambda)(a_3 - \lambda)} & \frac{1}{(a_3 - \lambda)}& 0 & \cdots \\
  \vdots & \vdots & \vdots & \vdots & \vdots
  \end{array} } \right)
\end{equation}
Let $C_k$ be the $\ell_1$ - norm of the $k$ - th column of the matrix $(\Delta_{a,b}-\lambda I)^{-1}$. Therefore we need to show that $\sup\limits_k C_k < \infty$. Now, from the above matrix,
\be
C_k=\frac{1}{|a_k - \lambda|} + \left| \frac{b_k}{(a_k - \lambda)(a_{k+1}- \lambda)} \right| +\left| \frac{b_kb_{k+1}}{(a_k - \lambda)(a_{k+1}- \lambda)(a_{k+2}- \lambda)} \right|+ \cdots ,
\ee
for $k \in \mathbb{N}$. By rearranging the terms (i.e. grouping multiple of $m$ terms) of $C_k$
%\bee
%C_k &=& \big\{1\mbox{ st term}+(m+1)\mbox{ th term}+(2m+1)\mbox{ th term}+\cdots\big\}\\
%&& +\big\{2\mbox{ nd term}+(m+2)\mbox{ th term}+(2m+2)\mbox{ th term}+\cdots\big\}\\
%&& +\cdots+ \big\{m\mbox{ th term}+2m\mbox{ th term}+3m\mbox{ th term}+\cdots\big\}
%\eee
 we have
{\fontsize{8}{7}\begin{eqnarray*}
C_k &=& \frac{1}{|a_k - \lambda|} \Bigg[ 1+ \left|  \frac{b_k b_{k+1}\cdots b_{k+m-1}}{(a_{k+1} - \lambda)(a_{k+2} - \lambda)\cdots (a_{k+m} - \lambda)} \right| \\
&& + \left| \Bigg( \frac{(b_k b_{k+1}\cdots b_{k+m-1})}{(a_{k+1} - \lambda)\cdots (a_{k+m} - \lambda)}\Bigg)\cdot \Bigg(\frac{(b_{k+m}\cdots b_{k+2m-1})}{(a_{k+m+1} - \lambda)\cdots (a_{k+2m} - \lambda)}\Bigg)\right|\\
&& +  \left| \Bigg( \frac{(b_k b_{k+1}\cdots b_{k+m-1})}{(a_{k+1} - \lambda)\cdots (a_{k+m} - \lambda)}\Bigg) \cdot \Bigg( \frac{(b_{k+m}\cdots b_{k+2m-1})}{(a_{k+m+1} - \lambda)\cdots (a_{k+2m} - \lambda)} \Bigg) \cdot \Bigg(\frac{(b_{k+2m}\cdots b_{k+3m-1})}{(a_{k+2m+1} - \lambda)\cdots (a_{k+3m} - \lambda)}\Bigg)\right|\\
&& + \cdots \Bigg] + \left| \frac{b_k}{(a_k - \lambda)(a_{k+1}- \lambda)} \right| \Bigg[ 1+ \left|  \frac{ b_{k+1}\cdots b_{k+m}}{(a_{k+2} - \lambda)\cdots (a_{k+m+1} - \lambda)} \right|\\
&& + \left| \Bigg( \frac{ (b_{k+1}\cdots b_{k+m})}{(a_{k+2} - \lambda)\cdots (a_{k+m+1} - \lambda)} \Bigg) \cdot \Bigg(\frac{(b_{k+m+1}\cdots b_{k+2m})}{(a_{k+m+2} - \lambda)\cdots (a_{k+2m+1} - \lambda)}\Bigg)\right| \\
&& +  \left|  \Bigg(\frac{(b_{k+1}\cdots b_{k+m})}{(a_{k+2} - \lambda)\cdots (a_{k+m+1} - \lambda)}\Bigg) \cdot \Bigg(\frac{(b_{k+m+1}\cdots b_{k+2m})}{(a_{k+m+2} - \lambda)\cdots (a_{k+2m+1} - \lambda)}\Big) \cdot \Bigg(\frac{(b_{k+2m+1}\cdots b_{k+3m})}{(a_{k+2m+2} - \lambda)\cdots (a_{k+3m+1} - \lambda)}\Bigg) \right| \\
&& + \cdots \Bigg] + \cdots + \left|  \frac{b_k b_{k+1}\cdots b_{k+m-2}}{(a_{k} - \lambda)(a_{k+1} - \lambda)\cdots (a_{k+m-1} - \lambda)} \right| \Bigg[ 1+ \left|  \frac{b_{k+m-1}\cdots b_{k+2m-2}}{(a_{k+m} - \lambda) \cdots (a_{k+2m-1} - \lambda)} \right| \\
&& + \left| \Bigg( \frac{(b_{k+m-1}\cdots b_{k+2m-2})}{(a_{k+m} - \lambda) \cdots (a_{k+2m-1} - \lambda)}\Bigg) \cdot \Bigg(\frac{(b_{k+2m-1}\cdots b_{k+3m-2})}{(a_{k+2m} - \lambda) \cdots (a_{k+3m-1} - \lambda)} \Bigg)\right| + \cdots \Bigg].
\end{eqnarray*}}
Further,
\begin{eqnarray*}
&\lim\limits_{n \rightarrow \infty}& \left| \frac{b_{k+nm}b_{k+nm+1}\cdots b_{k+(n+1)m-1}}{(a_{k+nm+1} - \lambda)(a_{k+nm+2} - \lambda) \cdots (a_{k+(n+1)m} - \lambda)} \right|\\
&& = \left| \frac{q_1q_2\cdots q_m}{(p_1 - \lambda)(p_2 - \lambda)\cdots (p_m - \lambda)} \right| < 1,\\
&\lim\limits_{n \rightarrow \infty}& \left| \frac{b_{k+nm+1}b_{k+nm+2}\cdots b_{k+(n+1)m}}{(a_{k+nm+2} - \lambda)(a_{k+nm+3} - \lambda) \cdots (a_{k+(n+1)m+1} - \lambda)} \right|\\
&& = \left| \frac{q_1q_2\cdots q_m}{(p_1 - \lambda)(p_2 - \lambda)\cdots (p_m - \lambda)} \right| < 1,\\
&& \hspace{2cm}\vdots \\
&\lim\limits_{n \rightarrow \infty}& \left| \frac{b_{k+(n+1)m-1}b_{k+(n+1)m}\cdots b_{k+(n+2)m-2}}{(a_{k+(n+1)m} - \lambda)(a_{k+(n+1)m+1} - \lambda) \cdots (a_{k+(n+2)m-1} - \lambda)} \right|\\
&& = \left| \frac{q_1q_2\cdots q_m}{(p_1 - \lambda)(p_2 - \lambda)\cdots (p_m - \lambda)} \right| < 1.
\end{eqnarray*}
shows that each series in R.H.S. is convergent, so $C_k$ is convergent for all $k\in\mathbb{N}$. Since
\begin{equation*}
\lim\limits_{k\to\infty} \Big| \frac{b_k b_{k+1} \cdots b_{k+(m-1)}}{(a_{k+1} - \lambda)(a_{k+2} - \lambda) \cdots (a_{k+m} - \lambda)} \Big| = \Big| \frac{q_1 q_2 \cdots q_m}{(p_1 - \lambda)(p_2 - \lambda)\cdots (p_m - \lambda)} \Big| < 1,
\end{equation*}
so there exists $N \in \mathbb{N}$ and a real number $\gamma < 1$ such that
\begin{equation*}
\Big| \frac{b_k b_{k+1} \cdots b_{k+(m-1)}}{(a_{k+1} - \lambda)(a_{k+2} - \lambda) \cdots (a_{k+m} - \lambda)} \Big|  < \gamma, \ \ \mbox{ for all } \ \forall \ \ k\geq N
\end{equation*}
Now for $k \geq N$
\begin{eqnarray*}
C_k &\leq & \frac{1}{|a_k - \lambda |} [1 + \gamma + \gamma ^2 + \cdots] + \left|\frac{b_k}{(a_k - \lambda)(a_{k+1} - \lambda)} \right| [1 + \gamma + \gamma ^2 + \cdots] + \cdots \\
&& + \left|  \frac{b_k b_{k+1}\cdots b_{k+m-2}}{(a_{k} - \lambda)(a_{k+1} - \lambda)\cdots (a_{k+m-1} - \lambda)} \right| [1 + \gamma + \gamma ^2 + \cdots].
\end{eqnarray*}
Again by Lemma \ref{lm3.1} there exists $N_0 \in \mathbb{N}$ and a real number $\alpha>0$ such that
for all $k \geq N_0$
\begin{eqnarray*}
&& \frac{1}{|a_k - \lambda |} \leq \alpha, \\
&& \left|\frac{b_k}{(a_k - \lambda)(a_{k+1} - \lambda)} \right| \leq \alpha,\\
&& \hspace{2cm}\vdots \\
&& \left|  \frac{b_k b_{k+1}\cdots b_{k+m-2}}{(a_{k} - \lambda)(a_{k+1} - \lambda)\cdots (a_{k+m-1} - \lambda)} \right| \leq \alpha.
\end{eqnarray*}
Thus $C_k \leq m \frac{\alpha}{1-\gamma}$ for all $k \geq \max\{N,N_0\}$. Therefore $\sup\limits_{k} C_k < \infty .$ This shows that $(\Delta_{a,b}-\lambda I)^{-1} \in (\ell_1 : \ell_1)$.\\
Next we show that $(\Delta_{a,b}-\lambda I)^{-1} \in (\ell_\infty : \ell_\infty)$ i.e. it is required to prove that supremum of $\ell_1$ norm of rows is finite. For this, let $\beta\in\mathbb{N}$. So $\beta = mk+j$ where $k \in \mathbb{N} \cup \{0\}$ and $j = 0,1, \cdots,m-1$. Let $R_{\beta}$ denote the $\beta$ - th row of $(\Delta_{a,b}-\lambda I)^{-1}$. Now for $\beta\geq m$, $R_{\beta} = \sum\limits_{r=1} ^{m} R_{mk+j}^{(r)}$ where $R_{\beta} ^{(r)}$ denotes the $r$ - th term of the series $R_{\beta}$ for $1 \leq r \leq m$.\\
The term $R_{mk+j} ^{(r)}$ is nothing but the following expressions.
\begin{enumerate}[(I)]
\item When $r=1$, \begin{enumerate}[(A)]
\item For $j\neq 0$,
{\fontsize{10}{9}\begin{eqnarray*}
&& R_{mk+j}^{(1)} = \frac{1}{|a_{km+j} - \lambda|}\left[ 1+ \left| \frac{b_{km-1+j}b_{km-1-1+j} \cdots b_{km-1-(m-1)+j}}{(a_{km-1+j}  -\lambda)(a_{km-1-1+j} - \lambda) \cdots (a_{km-1-(m-1)+j}- \lambda)} \right| \right. \\
&& + \left| \frac{b_{km-1+j}\cdots b_{km-1-(m-1)+j} b_{m(k-1)-1+j} \cdots b_{m(k-1)-1-(m-1)+j}}{(a_{km-1+j}  -\lambda) \cdots(a_{km-1-(m-1)+j} - \lambda) (a_{m(k-1)-1+j} \lambda) \cdots (a_{m(k-1)-1-(m-1)+j}- \lambda)} \right| \\
&&  + \cdots + \left| \frac{b_{km-1+j}\cdots b_{km-1-(m-1)+j} b_{m(k-1)-1+j} \cdots b_{m(k-1)-1-(m-1)+j}\cdots}{(a_{km-1+j}  -\lambda) \cdots(a_{km-1-(m-1)+j} - \lambda)\cdots} \right.\\
&& \left. \left. \times \frac{ b_{m(k-(k-2))-1+j}\cdots b_{m(k-(k-2))-1-(m-1)+j}b_{m(k-(k-1))-1+j}\cdots b_{m(k-(k-1))-1-(m-1)+j}}{ \cdots (a_{m(k-(k-1))-1-(m-1)+j}- \lambda)} \right| \right]
\end{eqnarray*}}
\item For $j=0$,
{\fontsize{8}{8}\begin{eqnarray*}
&& R_{mk}^{(1)} = \frac{1}{|a_{mk}-\lambda|}\Bigg[ 1 + \left| \frac{b_{km-1}b_{km-1-1} \cdots b_{km-1-(m-1)}}{(a_{km-1}  -\lambda)(a_{km-1-1} - \lambda) \cdots (a_{km-1-(m-1)}- \lambda)} \right| \\
&&+ \left| \frac{b_{km-1} \cdots b_{km-1-(m-1)}b_{m(k-1)-1}\cdots b_{m(k-1)-1-(m-1)}}{(a_{km-1}  -\lambda) \cdots (a_{km-1-(m-1)}  -\lambda)(a_{m(k-1)-1}  -\lambda)\cdots (a_{m(k-1)-1-(m-1)}- \lambda)} \right| + \cdots \\
&&+ \left| \frac{b_{km-1} \cdots b_{km-1-(m-1)} \cdots b_{m(k-(k-2))-1} \cdots b_{m(k-(k-2))-1-(m-1)}}{(a_{km-1}  -\lambda) \cdots (a_{km-1-(m-1)}  -\lambda) \cdots (a_{m(k-(k-2))-1}  -\lambda) \cdots (a_{m(k-(k-2))-1-(m-1)}- \lambda)} \right|\Bigg]
\end{eqnarray*}}
\end{enumerate}
\item when $r \leq j$ then
{\fontsize{10}{9}\begin{eqnarray*}
&&R_{mk+j}^{(r)} =  \left| \frac{b_{km-1+j}b_{km-1-1+j} \cdots b_{km-(r-1)+j}}{(a_{km+j}  -\lambda)(a_{km-1+j} - \lambda) \cdots (a_{km-(r-1)+j}- \lambda)} \right| \Bigg[ 1 + \\
&& \left| \frac{b_{km-r+j}b_{km-r-1+j} \cdots b_{km-r-(m-1)+j}}{(a_{km-r+j}  -\lambda)(a_{km-r-1+j} - \lambda) \cdots (a_{km-r-(m-1)+j}- \lambda)} \right| \\
&&+ \left| \frac{b_{km-r+j}b_{km-r-1+j} \cdots b_{km-r-(m-1)+j}b_{m(k-1)-r+j}\cdots b_{m(k-1)-r-(m-1)+j}}{(a_{km-r+j}  -\lambda) \cdots (a_{m(k-1)-r-(m-1)+j}- \lambda)} \right| + \cdots \\
&&+ \left| \frac{b_{km-r+j} \cdots b_{km-r-(m-1)+j} \cdots b_{m(k-(k-2))-r+j} \cdots b_{m(k-(k-2))-r-(m-1)+j}}{(a_{km-r+j}  -\lambda) \cdots (a_{m(k-(k-2))-r-(m-1)+j}- \lambda)} \right| \\
&&+ \left| \frac{b_{km-r+j}\cdots b_{km-r-(m-1)+j} b_{m(k-1)-r+j} \cdots b_{m(k-1)-r-(m-1)+j}\cdots}{(a_{km-r+j}  -\lambda) \cdots(a_{km-r-(m-1)+j} - \lambda)\cdots} \right.\\
&& \left. \times \frac{ b_{m(k-(k-2))-r+j}\cdots b_{m(k-(k-2))-r-(m-1)+j}b_{m(k-(k-1))-r+j}\cdots b_{m(k-(k-1))-r-(m-1)+j}}{ \cdots (a_{m(k-(k-1))-r-(m-1)+j}- \lambda)} \right|\Bigg].
\end{eqnarray*}}
\item and when $r > j$ then
{\fontsize{10}{9}\begin{eqnarray*}
&&R_{mk+j}^{(r)} =  \left| \frac{b_{km-1+j}b_{km-1-1+j} \cdots b_{km-(r-1)+j}}{(a_{km+j}  -\lambda) \cdots (a_{km-(r-1)+j}- \lambda)} \right| \Bigg[ 1 + \\
&& \left| \frac{b_{km-r+j}b_{km-r-1+j} \cdots b_{km-r-(m-1)+j}}{(a_{km-r+j}  -\lambda) \cdots (a_{km-r-(m-1)+j}- \lambda)} \right| \\
&&+ \left| \frac{b_{km-r+j}b_{km-r-1+j} \cdots b_{km-r-(m-1)+j}b_{m(k-1)-r+j}\cdots b_{m(k-1)-r-(m-1)+j}}{(a_{km-r+j}  -\lambda) \cdots (a_{m(k-1)-r-(m-1)+j}- \lambda)} \right| + \cdots \\
&&+  \left| \frac{b_{km-r+j} \cdots b_{km-r-(m-1)+j} \cdots b_{m(k-(k-2))-r+j} \cdots b_{m(k-(k-2))-r-(m-1)+j}}{(a_{km-r+j}  -\lambda) \cdots (a_{m(k-(k-2))-r-(m-1)+j}- \lambda)} \right|\Bigg].
\end{eqnarray*}}
\end{enumerate}
Since each series $R_k$ is finite for all $k\in\mathbb{N}$, so $R_k$ is convergent for all $k\in\mathbb{N}$ and since
\[\lim\limits_{t \rightarrow \infty} \left|  \frac{b_t b_{t+1} \cdots b_{t+(m-1)}}{(a_t - \lambda) (a_{t+1} - \lambda) \cdots (a_{t+(m-1)}-\lambda)}\right| = \left|  \frac{q_1 q_2 \cdots q_m}{(p_1 - \lambda) (p_2 - \lambda) \cdots (p_m-\lambda)}\right| < 1.\]
So there exists $N_1 \in \mathbb{N}$ with $N_1>m$ and two numbers $\gamma < 1$ and $M_{N_1}\geq 1$ such that for $t \geq N_1$
\[
\frac{b_tb_{t+1}\cdots b_{t+(m-1)}}{(a_t-\lambda)(a_{t+1}-\lambda) \cdots (a_{t+(m-1)}-\lambda)} < \gamma.
\]
If $N_1=mk+j$ with $k\in\mathbb{N}$, $j\in\{0,1,2,\cdots, m-1\}$ and $r\in\{1,2, \cdots , m\}$, then $M_{N_1}=\max\limits_r M_{N_1} ^{(r)}$ where
\begin{eqnarray*}
M_{N_1} ^{(r)} &=& \begin{cases}
1 + \sum_{n=0} ^k \prod_{l=0} ^n P_r ^{l+1}, & \mbox{ if } j-r>0\\
1 + \sum_{n=0} ^{k-1} \prod_{l=0} ^n P_r ^{l+1}, & \mbox{ if } j-r\leq 0
\end{cases}\\
\mbox{and } \ P_r ^{l+1} &=& \bigg| \frac{b_{m(k-l)+j-r+(m-1)} \cdots b_{m(k-l)+j-r}}{(a_{m(k-l)+j-r+(m-1)}-\lambda) \cdots (a_{m(k-l)+j-r}-\lambda)} \bigg|.
\end{eqnarray*}
Now for any $t\geq N_1$,
\begin{eqnarray*}
R_{t} ^{(1)} &\leq& \big| \frac{1}{(a_{t}-\lambda)} \big| \frac{M_{N_1}}{1-q_0}\\
\mbox{and } \ R_{t}^{(r)} &\leq& \left|  \frac{b_{t-1}  \cdots b_{t-(r-1)}}{(a_{t} - \lambda) \cdots (a_{t-(r-1)}-\lambda)}\right|. \frac{M_{N_1}}{1-q_0} \ \mbox{ for } \ r>1
\end{eqnarray*}
Again from Lemma \ref{lm3.1} there exists $N_0 \in \mathbb{N}$ and a real number $\alpha >0$ such that for $(t -(m-1))\geq N_0$,
\begin{eqnarray*}
&& \frac{1}{|a_{t} - \lambda |} \leq \alpha, \\
&& \left|\frac{b_{t-1}}{(a_{t} - \lambda)(a_{t-1} - \lambda)} \right| \leq \alpha, \\
&& \hspace{2cm} \vdots \\
&& \left|  \frac{b_{t-(m-2)} \cdots b_{t-1}b_t}{(a_{t-(m-1)} - \lambda)(a_{t-(m-2)} - \lambda)\cdots (a_{t-1} - \lambda)(a_t-\lambda)} \right|\leq \alpha.
\end{eqnarray*}
Hence for $t\geq \max\{N_1,\big(N_0+(m-1)\big)\}$,
\[R_t ^{(r)} \leq \frac{\alpha M_{N_1}}{1-\gamma}.\]
Thus for $t\geq \max\{N_1,\big(N_0+(m-1)\big)\}$,
\[R_t \leq \Big(\frac{\alpha M_{N_1}}{1-\gamma}\Big) \ m.\]
Therefore $\sup\limits_{t} R_{t} < \infty.$ This shows that $(\Delta_{a,b}-\lambda I)^{-1} \in (\ell_\infty : \ell_\infty)$. Hence $\sigma(\Delta_{a,b}, \ell_p) \subseteq S_1 \cup S_2$.\\
Further, suppose that $\lambda \notin \sigma(\Delta_{a,b}, \ell_p)$ i.e. $\lambda$ is not in the spectrum of $\Delta_{a,b}$. Then $(\Delta_{a,b}-\lambda I)^{-1} \in B(\ell_p)$.
%If $\lambda=a_k$, then $a_k \in \sigma_p (\Delta_{a,b} ^*, \ell_p ^*)$ [shown in Theorem \ref{point_p*}]. Hence by Lemma \ref{L1.7.9}, the range of $(\Delta_{a,b}-a_k I)$ is not dense in $\ell_p$ which implies $\lambda \in \sigma(\Delta_{a,b}, \ell_p)$.
So, $\lambda\notin \sigma(\Delta_{a,b}, \ell_p)$ implies $\lambda\neq a_k$ for all $k\in\mathbb{N}$ and the operator $(\Delta_{a,b}-\lambda I)^{-1}$ transforms the unit sequence $e = (1, 0, 0, \cdots)^t$ into a sequence in $\ell_p$, so we have
\begin{equation*}
(\Delta_{a,b}-\lambda I)^{-1}e = \left( \frac{1}{a_1 - \lambda }, \frac{-b_1}{(a_1 - \lambda)(a_{2} - \lambda)},  \frac{b_1b_2}{(a_1 - \lambda)(a_{2} - \lambda)(a_3 - \lambda)}, \cdots \right)^t \in l_p.
\end{equation*}
Which implies
{\fontsize{10}{10}\bea
\frac{1}{|a_1 - \lambda|^p }+ \left|\frac{b_1}{(a_1 - \lambda)(a_{2} - \lambda)}\right|^p +\left| \frac{b_1b_2}{(a_1 - \lambda)(a_{2} - \lambda)(a_3 - \lambda)}\right|^p + \cdots < \infty
\eea
which gives
\begin{equation*}
\lambda \notin \Bigg\{ \lambda \in \mathbb{C} : \big| q_1q_2\cdots q_m \big| > \big|(p_1 - \lambda)(p_2 - \lambda)\cdots (p_m - \lambda) \big| \Bigg\}
\end{equation*}
otherwise,
\begin{equation*}
 \big| q_1q_2\cdots q_m \big| > \big| (p_1 - \lambda)(p_2 - \lambda)\cdots (p_m - \lambda) \big| \ \ \mbox{ since } \ \ 1<p<\infty
\end{equation*}
}
and hence the following series
{\fontsize{9}{9}
\[ \frac{1}{|a_1-\lambda|^p} + \Big|\frac{b_1b_2\cdots b_m}{(a_1-\lambda)(a_2-\lambda) \cdots (a_{m+1}-\lambda)}\Big|^p + \Big|\frac{b_1b_2\cdots b_m b_{m+1}\cdots b_{2m}}{(a_1-\lambda)(a_2-\lambda) \cdots (a_{m+1}-\lambda)(a_{m+2}-\lambda)(a_{2m+1}-\lambda)}\Big|^p + \cdots\]}
is divergent. So by comparison test  $(\Delta_{a,b}-\lambda I)^{-1}e$ is divergent, which is a contradiction. Hence
\[ \Bigg\{ \lambda \in \mathbb{C} : \big| q_1q_2\cdots q_m \big| > \big| (p_1 - \lambda)(p_2 - \lambda)\cdots (p_m - \lambda) \big| \Bigg\} \subseteq \sigma(\Delta_{a,b}, \ell_p) \]
Since $\sigma(\Delta_{a,b}, \ell_p)$ is compact, so
\[ \Bigg\{ \lambda \in \mathbb{C} : \big| (p_1 - \lambda)(p_2 - \lambda)\cdots (p_m - \lambda) \big| \leq \big|q_1q_2\cdots q_m\big| \Bigg\} \subseteq \sigma(\Delta_{a,b}, \ell_p) \]
Thus $S_1 \cup S_2 \subseteq \sigma(\Delta_{a,b}, \ell_p)$. This completes the proof.
\end{proof}

\begin{lemma}\label{lemma3.2}
Assume $a_j\neq a_{j+k}$ for all $k\in\mathbb{N}$. Then $a_j\in S_2 \cup S_3$ if and only if the following condition holds
$$\sum\limits_{i = j} ^{\infty} \left| \frac{b_j b_{j+1} \cdots b_{i}}{(a_{j+1} - a_j)(a_{j+2}- a_j) \cdots (a_{i+1} - a_j)} \right|^p< \infty$$
\end{lemma}
\begin{proof}
By grouping multiple of $m$ terms, one can easily show that if $a_j\in S_2$, then
\begin{eqnarray*}
&& \sum\limits_{i = j} ^{\infty} \left| \frac{b_j b_{j+1} \cdots b_{i}}{(a_{j+1} - a_j)(a_{j+2}- a_j) \cdots (a_{i+1} - a_j)} \right|^p< \infty\\
\mbox{i.e. } && a_j \in S_2 \cup S_3 \implies \sum\limits_{i = j} ^{\infty} \left| \frac{b_j b_{j+1} \cdots b_{i}}{(a_{j+1} - a_j)(a_{j+2}- a_j) \cdots (a_{i+1} - a_j)} \right|^p< \infty
\end{eqnarray*}
Conversely, let us assume $\Big|\frac{(a_j-p_1)(a_j-p_2)\cdots (a_j-p_m)}{q_1 q_2\cdots q_m}\Big| < 1 $, then one can easily prove that the series
\[\sum\limits_{i = j}^{\infty} \left| \frac{b_j b_{j+1} \cdots b_{i}}{(a_{j+1} - a_j)(a_{j+2}- a_j) \cdots (a_{i+1} - a_j)} \right|^p\] is divergent. Thus we have
\[ \sum\limits_{i = j}^{\infty} \left| \frac{b_j b_{j+1} \cdots b_{i}}{(a_{j+1} - a_j)(a_{j+2}- a_j) \cdots (a_{i+1} - a_j)} \right|^p <\infty \ \mbox{ if and only if } \ a_j\in S_2\cup S_3 .\]
\end{proof}

\begin{theorem}\label{point_p}
$\sigma_p(\Delta_{a,b}, l_p)=S=S_2 \cup S_3$.
\end{theorem}
\begin{proof}
Suppose $\Delta_{a,b}x = \lambda x$ for any $x \in l_p$. Then we have the following system of linear equations
\begin{eqnarray*}
a_1 x_1 &= &\lambda x_1 \\
b_1 x_1 + a_2 x_2 &=& \lambda x_2 \\
b_2 x_2 + a_3 x_3 &=& \lambda x_3\\
& \vdots
\end{eqnarray*}
From the above equations, we obtain
\begin{eqnarray*}
(a_1 -\lambda)x_1 &=& 0\\
\mbox{and } \ b_kx_k+(a_{k+1} - \lambda)x_{k+1} &=& 0, \ \mbox{ for all } \ k \in \mathbb{N}.
\end{eqnarray*}
Now for all $\lambda \notin \{a_k: k \in \mathbb{N}\},$ we have $x_k=0$ for all $k \in \mathbb{N}$. So for $\lambda\neq a_k$ for all $k \in \mathbb{N}$, there does not exist any $x\neq\theta\in\ell_p$ such that $\Delta_{a,b}x = \lambda x$ holds. Thus if $\lambda\neq a_k$ for all $k\in\mathbb{N}$, then $\lambda \notin \sigma_p(\Delta_{a,b}, l_p)$. Consequently $\sigma_p(\Delta_{a,b}, l_p) \subseteq \{a_k: k \in \mathbb{N}\}$.\\
Now we consider two cases.\\
\textbf{\underline{Case 1} :} Let $\{a_k\}$ is a periodic sequence of period $m$. Which implies
\[ a_{km-i} = a_{m-i} , \ \ \forall \ \ k \in \mathbb{N} \ \ \mbox{and} \ \ i = 0,1, \cdots, m-1. \]
Now for any $x\in\ell_p$, $\Delta_{a,b}x = \lambda x$ implies
\begin{eqnarray*}
a_1 x_1 &= &\lambda x_1 \\
b_1 x_1 + a_2 x_2 &=& \lambda x_2 \\
& \vdots &\\
b_{m-1} x_{m-1} + a_{m} x_{m} &=& \lambda x_{m}\\
& \vdots &\\
b_{km} x_{km} + a_{1} x_{km+1} &=& \lambda x_{km+1}\\
b_{km+1} x_{km+1} + a_{2} x_{km+2} &=& \lambda x_{km+2}\\
& \vdots &
\end{eqnarray*}
Since $\sigma_p(\Delta_{a,b}, l_p) \subseteq \{a_k: k \in \mathbb{N}\}$ so $\sigma_p(\Delta_{a,b}, l_p) \subseteq \{a_1, a_2, \cdots, a_m\}$ as $\{a_k\}$ is periodic sequence with period $m$.\\
Suppose $\lambda=a_i$ is an eigenvalue of $\Delta_{a,b}$ where $i\in\{1,2,\ldots, m\}$. So there exists a non zero $x\in\ell_p$ such that $\Delta_{a,b} x = a_i x$ hold. Let $x_r$ be the first non-zero entry of  $x=\{x_k\}$.\\
Now there always exists a number $t=i+km$ where $k\in\mathbb{N}$ such that $t>r$. So from the equation
\begin{eqnarray*}
b_{t-1} x_{t-1} + a_t x_t &=& a_i x_t\\
\mbox{i.e. } \ b_{i+km-1} \ x_{i+km-1} + a_{i+km} \ x_{i+km} &=& a_i \ x_{i+km} \ \mbox{ implies } \ x_{i+km-1}=0
\end{eqnarray*}
as $b_{i+km-1}\neq 0$ and $a_{i+km}=a_i$. Also from the equations
\[ b_{i+km-2} \ x_{i+km-2} + a_{i+km-1} \ x_{i+km-1} = a_i \ x_{i+km-1}\]
gives $x_{i+km-2} = 0$, since $b_{i+km-2}\neq 0$ and also from above, we get $x_{i+km-1}=0$\\
Continuing this process we get
\bc
$x_{t-1} = x_{t-2} = \cdots = x_r =0$\\
\ec
which contradicts the fact $x_r\neq 0$.\\
Hence $\lambda=a_i \notin \sigma_p (\Delta_{a,b}, \ell_p)$ for $i\in\{1,2,\ldots, m\}$.\\
So $\sigma_p (\Delta_{a,b}, \ell_p)=\emptyset$.\\
\textbf{\underline{Case 2} :} Let $\{a_k\}$ is a sequence of real numbers such that
\bc
$a_{km-i} \rightarrow p_{m-i} \ \ \ \mbox{as} \ \ k \rightarrow \infty \ \ \mbox{for} \ \ i = 0,1, \cdots, m-1$.\\
\ec
and we know that $\sigma_p (\Delta_{a,b}, \ell_p) \subseteq \big\{ a_k : k\in\mathbb{N}\big\} $. Also if $\lambda=p_i$ for $i\in\{1,2,\ldots, m\}$, then one can easily prove that $\lambda\notin \sigma_p (\Delta_{a,b}, \ell_p)$. So $\sigma_p (\Delta_{a,b}, \ell_p)\subseteq \big\{ a_k : k\in\mathbb{N}\big\}\setminus \{p_1,p_2,\ldots,p_m\}$.\\
Now we prove that $\lambda\in \sigma_p (\Delta_{a,b}, \ell_p)$  if and only if $\lambda\in S$.\\
If $\lambda\in \sigma_p (\Delta_{a,b}, \ell_p)$, then $\lambda=a_j\neq p_i$ for some $j\in\mathbb{N}$ and $i\in\{1,2,\ldots, m\}$ and there exists a non zero $x\in \ell_p$ such that $\Delta_{a,b}x = a_j x$ holds.\\
Again from the equation $\Delta_{a,b}x = \lambda x$ we have
\begin{eqnarray*}
a_1 x_1 &= &\lambda x_1 \\
b_1 x_1 + a_2 x_2 &=& \lambda x_2 \\
& \vdots &\\
b_{m-1} x_{m-1} + a_{m} x_{m} &=& \lambda x_{m}\\
b_{m} x_{m} + a_{m+1} x_{m+1} &=& \lambda x_{m+1}\\
& \vdots &
\end{eqnarray*}
Since $\lambda=a_j\neq p_i$ for $i\in\{1,2,\cdots, m\}$, so there exists a natural number $s\geq j$ such that $a_s=a_j$ and for $n>s$, $a_n\neq a_j$. Clearly, for $\lambda=a_s$, the above equations give $x_1 = x_2 = \cdots = x_{s-1} = 0$ (as $b_k\neq 0$ for all $k\in\mathbb{N}$). Thus the system reduces to
\begin{eqnarray*}
a_s x_s &= &a_s x_s \\
b_s x_s + a_{s+1} x_{s+1} &=& a_s x_{s+1} \\
b_{s+1} x_{s+1} + a_{s+2} x_{s+2} &=& a_s x_{s+2}\\
& \vdots & \\
\mbox{which implies }
|x_l| &=& \left| \frac{b_s b_{s+1} \cdots b_{l-1}}{(a_{s+1} - a_s)(a_{s+2}- a_s) \cdots (a_l - a_s)} \right| |x_s|, \ \ \mbox{for} \ \ l>s.
\end{eqnarray*}
Now, for a non zero $x$, if we choose $x_s\neq 0$, then $a_s \in \sigma_p(\Delta_{a,b}, \ell_p)$ if and only if
\begin{equation*}
\sum\limits_{i = s}^{\infty} \left| \frac{b_s b_{s+1} \cdots b_{i}}{(a_{s+1} - a_s)(a_{s+2}- a_s) \cdots (a_{i+1} - a_s)} \right|^p < \infty .
\end{equation*}
Again by Lemma \ref{lemma3.2},
\[ \sum\limits_{i = s}^{\infty} \left| \frac{b_s b_{s+1} \cdots b_{i}}{(a_{s+1} - a_s)(a_{s+2}- a_s) \cdots (a_{i+1} - a_s)} \right|^p < \infty \ \mbox{ if and only if } \ a_s\in S
\]
Since $a_j=a_s$, so $a_j \in \sigma_p(\Delta_{a,b}, \ell_p)$ if and only if $a_j \in S$.
\end{proof}

\begin{theorem}\label{point_p*}
$\sigma_p(\Delta_{a,b}^*, l_p^*) = S_2 \cup S_4 \cup S_6$
\end{theorem}
\begin{proof}
We consider $\Delta_{a,b}^* x = \lambda x$. This gives
\begin{equation}\label{1}
a_k x_k + b_k x_{k+1} = \lambda x_{k}, \ \ \mbox{for} \ \ k \in \mathbb{N}.
\end{equation}
Consider the following cases.
\begin{enumerate}[(i)]
\item If $\lambda=a_1$, then putting $k=1$ in Equation \ref{1}, we get,
\begin{eqnarray*}
a_1 x_1 + b_1 x_2 &=& a_1 x_1 \\
\mbox{i.e. } \ x_2 &=& 0 \ \mbox{ since } \ b_1\neq 0
\end{eqnarray*}
Similarly, if put $k=2$ and $\lambda=a_1$ in Equation \ref{1}, we have
\[ a_2 x_2 + b_2 x_3 = a_1 x_2 \]
Since from above, $x_2=0$ and $b_2\neq0$, we get $x_3=0$.\\
Proceeding in a similar way, we can say $x_4=x_5=\cdots=0$. Hence if we choose $x_1\neq 0$, then there exists a non-zero $x\in \ell_q\cong \ell_p ^*$ such that $\Delta_{a,b}^* x = a_1 x$ holds. Thus $a_1\in \sigma_p(\Delta_{a,b}^*, l_p^*)$.
\item If $\lambda = a_j$ for any $j \in \mathbb{N}\setminus \{1\}$ and $k=j$, then Equation \ref{1}, reduces to
\begin{eqnarray*}
a_j \ x_j + b_j \ x_{j+1} &=& a_j \ x_{j} \\
\Rightarrow x_{j+1} &=& 0 \ \mbox{ since $b_j\neq 0$}.
\end{eqnarray*}
And putting $k=j+1$ and $\lambda=a_j$, we get
\begin{equation*}
a_{j+1} \ x_{j+1} + b_{j+1} \ x_{j+2} = a_{j} \ x_{j+1}
\end{equation*}
Since $b_{j+1}\neq 0$ and $x_{j+1}=0$, so we have $x_{j+2} = 0$. Hence one can conclude $x_{j+1} = x_{j+2} = \cdots = 0.$ Therefore for $\lambda = a_j$ we have
\begin{eqnarray*}
(a_1 - a_j)x_1 + b_1 x_2 &=& 0\\
(a_2 - a_j)x_2 + b_2 x_3 &=& 0\\
& \vdots &\\
(a_{j-1} - a_j)x_{j-1} + b_{j-1} x_j &=& 0.
\end{eqnarray*}
This is a system of $(j-1)$ homogeneous equations with $j$ variables. Hence at least one non-trivial solution exists.\\
This implies $a_j \in \sigma_p(\Delta_{a,b}^*, \ell_p^*)$ for all $j \in \mathbb{N}\setminus \{1\}$.
\item Now let $\lambda \neq a_j$ for all $j \in \mathbb{N}$. Then from Equation \ref{1}, we have
\begin{eqnarray*}
x_k &=& \frac{(\lambda - a_1)(\lambda - a_2) \cdots (\lambda - a_{k-1})}{b_1 b_2 \cdots b_{k-1}} x_1, \ \ \mbox{for} \ \ k=2,3, \cdots .\\
&=& f_k \ x_1  \ \ \mbox{for} \ \ k=2,3, \cdots ,
\end{eqnarray*}
where $f_k=\frac{(\lambda - a_1)(\lambda - a_2) \cdots (\lambda - a_{k-1})}{b_1 b_2 \cdots b_{k-1}}$.\\
Let $\lambda\in \sigma_p(\Delta_{a,b}^*, \ell_p^*)$. Which implies there exists $x\neq \theta$ in $l_p^*$ such that $\Delta_{a,b}^* x=\lambda x$ holds. So for any non zero $x$, if we choose $x_1\neq 0$, then $x\in \ell_p ^* \cong \ell_q$ if and only if the series
\[ \sum\limits_{k=2} ^{\infty} \left|\frac{(\lambda-a_1)(\lambda-a_2)\cdots (\lambda-a_{k-1})}{b_1 b_2 \cdots b_{k-1}}\right|^q < \infty \ \mbox{ i.e. } \ \sum\limits_{k=2} ^{\infty} |f_k|^q<\infty.\]
So $\lambda\in \sigma_p(\Delta_{a,b}^*, \ell_p^*)$ if and only if $\sum\limits_{k=2} ^{\infty} |f_k|^q<\infty$.
Further, if $\lambda\neq a_k$ for $k\in\mathbb{N}$ and $\lambda\in S_4$ then we can easily prove that
$\sum\limits_{k=2} ^{\infty} |f_k|^q<\infty$.
\end{enumerate}
Hence from Case (i), (ii) and (iii), one can conclude that $S_2 \cup S_4 \cup S_6\subseteq \sigma_p(\Delta_{a,b}^*, \ell_p^*)$.\\
Conversely, if $\lambda\notin S_2 \cup S_4 \cup S_6$, then we can easily prove that the series $\sum\limits_{k=2} ^{\infty} |f_k|^q$ is always divergent. So $\lambda\notin \sigma_p(\Delta_{a,b}^*, \ell_p^*)$ i.e. $\sigma_p(\Delta_{a,b}^*, \ell_p^*)\subseteq S_2 \cup S_4 \cup S_6$.\\
Hence one can conclude that $\sigma_p(\Delta_{a,b}^*, \ell_p^*)= S_2 \cup S_4 \cup S_6$.
This completes the proof.
\end{proof}

\begin{theorem}\label{res_p}
$\sigma_r(\Delta_{a,b}, \ell_p)=S_4\cup (S_6\setminus S_3)$.
\end{theorem}
\begin{proof}
From Lemma \ref{denserange} it can be easily proved that, $\sigma_r(\Delta_{a,b}, \ell_p) = \sigma_p(\Delta_{a,b}^*, \ell_p^*) \setminus \sigma_p(\Delta_{a,b}, \ell_p).$ Hence the result follows from Theorem \ref{point_p} and Theorem \ref{point_p*}.
\end{proof}

\begin{theorem}\label{cont_p}
$\sigma_c(\Delta_{a,b}, \ell_p) = S_5\setminus S_6$.
\end{theorem}
\begin{proof}
Since $\sigma_p(\Delta_{a,b}, \ell_p) \cup \sigma_r(\Delta_{a,b}, \ell_p) \cup \sigma_c(\Delta_{a,b}, \ell_p) = \sigma(\Delta_{a,b}, \ell_p)$, and all are mutually disjoint, so the result follows from Theorem \ref{sepct_p}, Theorem \ref{point_p} and Theorem \ref{res_p}.
\end{proof}

\begin{theorem}\label{twoband}
Let $\{a_k\}$ and $\{b_k\}$ are two sequences of real numbers with $b_k\neq0$ for all $k\in\mathbb{N}$ and
\bee
&& a_{2k-1} \rightarrow p_1, \ a_{2k} \rightarrow p_2 \\
 && b_{2k-1} \rightarrow q_1, \ b_{2k} \rightarrow q_2
\eee
as $k \rightarrow \infty$, where $q_i\neq0$ for $i=1,2$. If there exists a natural number $N$ such that
\bc
$|q_1q_2| - |b_kb_{k+1}| \geq 2R(|p_i - a_k| + |p_j - a_{k+1}|)$\\
\ec
holds for all $k\geq N$, where $i,j\in \{1,2\}$ and $i\neq j$ then for some $N_0\geq N$
\[\frac{|z - a_k| |z - a_{k+1}|}{|b_k b_{k+1}|}\geq 1\ \mbox{ for all } k\geq N_0\]
and for all $z$ in the set $\{z \in \mathbb{C} : |z - p_1||z- p_2| = |q_1q_2|\}$ with $z\neq a_k$ for all $k\in\mathbb{N}$.
Here $R \geq \sup_k|a_k| + \sup_k|b_k|$ be any real number.
\end{theorem}
\begin{proof}
Let $z\in\big\{z \in \mathbb{C} : |z - p_1||z- p_2| = |q_1q_2|\big\}$. If we choose $\varepsilon>0$ in such a way that $\varepsilon<\min\limits_i \mbox{dist}  \{p_i, \partial S_1\}$ where $\partial S_1=|z-p_1| \ |z-p_2| = |q_1 q_2|$, then for $\varepsilon>0$, there exists $N_0(\geq N)\in\mathbb{N}$ such that for $k\geq N_0$,
\[ |p_i-a_k|<\varepsilon \mbox{ and } |p_j-a_{k+1}|<\varepsilon. \]
Since by our assumption $\varepsilon<\min\limits_i \mbox{dist}  \{p_i, \partial S_1\}$ so
\[\varepsilon<|z-p_i| \mbox{ as well as } \varepsilon<|z-p_j|, \]
where $i,j\in \{1,2\}$ and $i\neq j$.
Hence for all $k\geq N_0$
\bee
|z-a_k| \ |z-a_{k+1}| &\geq& \big(|z-p_i| - |p_i-a_k|\big) \ \big(|z-p_j| - |p_j-a_{k+1}|\big)\\
&=& |z-p_i| \ |z-p_j| - |z-p_i| \ |p_j-a_{k+1}| - |z-p_j| \ |p_i-a_k|\\
&& + |p_i-a_k| \ |p_j-a_{k+1}| \\
&\geq& |q_1q_2| - |z-p_i| \ |p_j-a_{k+1}| -|z-p_j| \ |p_i-a_k|\\
&\geq&  |q_1q_2| - 2R \big( |p_j-a_{k+1}| + |p_i-a_k| \big)\\
&\geq& |b_k \ b_{k+1}|
\eee
Since $R\geq \sup\limits_k |a_k| + \sup\limits_k |b_k|$. So $2R \geq |z-p_i|$ as well as $2R \geq |z-p_j|$ because $R\geq \|\Delta_{a,b}\|_p$.\\
Hence for $k\geq N_0$, $\frac{|z-a_k| \ |z-a_{k+1}|}{|b_k \ b_{k+1}|} \geq 1$. This proves the result.
\end{proof}

\begin{corollary}\label{cor3.1}
If the condition of the theorem \ref{twoband} holds and $z\neq a_k$ for all $k\in\mathbb{N}$, then for all $z$ in $\big\{z\in\mathbb{C} :|z - p_1||z- p_2| = |q_1q_2| \big\}$, the series
\[\sum\limits_{k=2} ^{\infty} \left|\frac{(z-a_1)(z-a_2)\cdots (z-a_{k-1})}{b_1 b_2 \cdots b_{k-1}}\right|^q\]
is divergent for every $q>1$.
\end{corollary}
\begin{proof}
Suppose there exists $N_0\in\mathbb{N}$ such that for $k\geq N_0$, $\frac{|z-a_k| \ |z-a_{k+1}|}{|b_k \ b_{k+1}|} \geq 1$. Now if we consider the series
\bee
\left| \frac{(z-a_1)(z-a_2)\cdots (z-a_{N_0 - 1})}{b_1 b_2 \cdots b_{N_0-1}}\right|^q \Bigg( 1 + \left| \frac{(z-a_{N_0})(z-a_{N_0+1})}{b_{N_0} b_{N_0+1}}\right|^q \\
+\left| \frac{(z-a_{N_0})(z-a_{N_0+1})(z-a_{N_0+2})(z-a_{N_0+3})}{b_{N_0} b_{N_0+1} b_{N_0+2} b_{N_0+3}}\right|^q +\cdots\Bigg)\\
\mbox{[Grouping multiple of 2 terms]}
\eee
The above series is divergent. Hence by comparison test, the series $\sum\limits_{k=2} ^{\infty} \left|\frac{(z-a_1)(z-a_2)\cdots (z-a_{k-1})}{b_1 b_2 \cdots b_{k-1}}\right|^q$ is divergent.
\end{proof}

\begin{theorem}\label{thm3.1}
If the conditions mentioned in the  Theorem \ref{twoband} hold in case of the operator $\Delta_{a,b}$ with $m=2$, then the followings hold good :
\begin{enumerate}[(i)]
\item $\sigma (\Delta_{a,b}, \ell_p) = S_1 \cup S_2$
\item $\sigma_p(\Delta_{a,b}, \ell_p) = S_2 \cup S_3$
\item $\sigma_r(\Delta_{a,b}, \ell_p) = S_4 \cup (S_6\setminus S_3)$
\item $\sigma_c(\Delta_{a,b}, \ell_p) = \big\{ \lambda :\lambda \neq a_k, |\lambda - p_1| \ |\lambda - p_2| = |q_1 \ q_2| \big\}$
\end{enumerate}
\end{theorem}
\begin{proof}
%By corollary \ref{cor3.1}, $\sigma_c(\Delta_{a,b}, \ell_p)=\big\{ \lambda : \lambda\neq a_k, |\lambda - p_1| \ |\lambda - p_2| = |q_1 \ q_2| \big\}$.
By Corollary \ref{cor3.1}, we have $S_6 = \{a_k : |a_k - p_1| \ |a_k - p_2| = |q_1 \ q_2|\}$. Therefore the results $(i), (ii), (iii)$ and $(iv)$ follow from Theorem \ref{sepct_p}, Theorem \ref{point_p}, Theorem \ref{res_p} and Theorem \ref{cont_p} respectively.
\end{proof}

\begin{theorem}\label{goldberg}
The operator $\Delta_{a,b}$ satisfies the following relations,
\begin{enumerate}[(i)]
\item[(a)]  $A_3 \sigma(\Delta_{a,b}, \ell_p)=B_3 \sigma(\Delta_{a,b}, \ell_p)=  \emptyset ,$
\item[(b)] $C_3 \sigma(\Delta_{a,b}, \ell_p) = \sigma_p(\Delta_{a,b}, \ell_p)$
\item[(c)] $B_2 \sigma(\Delta_{a,b}, \ell_p) = \sigma_c(\Delta_{a,b}, \ell_p),$
\item[(d)] $C_1 \sigma(\Delta_{a,b}, \ell_p) \cup C_2 \sigma(\Delta_{a,b}, \ell_p) = \sigma_r(\Delta_{a,b}, \ell_p)$
\end{enumerate}
\end{theorem}
\begin{proof}
From Remark \ref{R1.7.7} we have the following relations
\begin{eqnarray*}
\sigma_p(\Delta_{a,b}, \ell_p) &=& A_3 \sigma(\Delta_{a,b}, \ell_p) \cup B_3 \sigma(\Delta_{a,b}, \ell_p) \cup C_3 \sigma(\Delta_{a,b}, \ell_p),\\
\sigma_r(\Delta_{a,b}, \ell_p) &=& C_1 \sigma(\Delta_{a,b}, \ell_p) \cup C_2 \sigma(\Delta_{a,b}, \ell_p),\\
\sigma_c(\Delta_{a,b}, \ell_p) &=& B_2 \sigma(\Delta_{a,b}, \ell_p).
\end{eqnarray*}
Since $\sigma_p(\Delta_{a,b}, \ell_p) \subset \sigma_p(\Delta_{a,b}^*, \ell_p^*),$ then by Lemma \ref{denserange} it follows that $\lambda \in \sigma_p(\Delta_{a,b}, l_p)$ if and only if $(\Delta_{a,b} - \lambda I)^{-1}$ does not exist and $\overline{R(\Delta_{a,b} - \lambda I)} \neq \ell_p$. This proves that $C_3 \sigma(\Delta_{a,b}, \ell_p) = \sigma_p(\Delta_{a,b}, l_p)$ and $A_3 \sigma(\Delta_{a,b}, \ell_p)=B_3 \sigma(\Delta_{a,b}, \ell_p)=  \emptyset$. The other two results follow from above relations.
\end{proof}

\begin{example}
Let $\Delta_{a,b}$ be an operator of the form \eqref{main} where
\begin{equation*}
a_k = \left\{
\begin{aligned}
&1 - \frac{1}{k^2}, \ k  = 1,3,5, \cdots &\\
&\frac{1}{2} - \frac{1}{k^2}, \ k  = 2,4,6, \cdots &
\end{aligned}
\right.  \ \mbox{and} \ \ \
b_k = \left\{
\begin{aligned}
&2 - \frac{1}{k}, \ k  = 1,3,5, \cdots &\\
&3 - \frac{1}{k}, \ k  = 2,4,6, \cdots &
\end{aligned}
\right.
\end{equation*}
Then $p_1 = 1, p_2 = \frac{1}{2}$ and $q_1 = 2, q_2 = 3$ and $R = \sup_k |a_k| + \sup_k|b_k| = 4.$ Also for odd $k,$
\[|q_1q_2| - |b_k b_{k+1}| = \frac{5k+2}{k(k+1)}. \]
Again
\[2R (|a_k - p_1| + |a_{k+1} - p_2|) =  \frac{16k^2 + 16k +8}{k^2(k+1)^2}. \]
Now
\begin{eqnarray*}
|q_1q_2| - |b_k b_{k+1}| &-&  2R (|a_k - p_1| + |a_{k+1} - p_2|) \\
&=& \frac{5k^3 - 9k^2-14k - 8}{k^2(k+1)^2} >0 \ \ \forall \ \ k \geq 4.
\end{eqnarray*}
Since for odd $k (\geq 5)$, $|q_1q_2| - |b_k b_{k+1}| > 2R (|a_k - p_1| + |a_{k+1} - p_2|)$, so if we choose $\varepsilon$ in a way such that $\varepsilon < dist (p_i, \partial S_1)$, where $\partial S_1= \big\{ z : |z-p_1| \ |z-p_2| = q_1 \ q_2 \big\}$ and $i\in\{1,2\}$, then by the technique as used in Theorem \ref{twoband}, there exists  $N_0\in\mathbb{N}$ such that for odd $k$ and $k\geq \max\{5, N_0\}$, we have  $\frac{|z-a_k| \ |z-a_{k+1}|}{|b_kb_{k+1}|} \geq 1$.\\
Now if we choose a odd number $M_0\geq\max\{5, N_0\}$, then for $q>1$, the series
\bee
\left| \frac{(z-a_1)(z-a_2)\cdots (z-a_{M_0 - 1})}{b_1 b_2 \cdots b_{M_0-1}}\right|^q \Bigg( 1 + \left| \frac{(z-a_{M_0})(z-a_{M_0+1})}{b_{M_0} b_{M_0+1}}\right|^q \\
+ \left| \frac{(z-a_{M_0})(z-a_{M_0+1})(z-a_{M_0+2})(z-a_{M_0+3})}{b_{M_0} b_{M_0+1} b_{M_0+2} b_{M_0+3}}\right|^q +\cdots\Bigg)
\eee
is divergent and hence the series $\sum\limits_{k=2} ^{\infty} \left|\frac{(z-a_1)(z-a_2)\cdots (z-a_{k-1})}{b_1 b_2 \cdots b_{k-1}}\right|^q$ is divergent.\\
Now the following results follow from theorem \ref{thm3.1},
\begin{enumerate}[(i)]
\item $\sigma(\Delta_{a,b}, l_p)=S_1 \cup S_2$
\item $\sigma_p(\Delta_{a,b}, l_p) = S_2 \cup S_3$
\item $\sigma_r(\Delta_{a,b}, l_p) = S_4 \cup (S_6\setminus S_3) $
\item $\sigma_c(\Delta_{a,b}, l_p) = \big\{ \lambda :\lambda \neq a_k, |\lambda - p_1| |\lambda - p_2| = |q_1 q_2| \big\}$.
\end{enumerate}
%\begin{figure}
%\centering
%\includegraphics[scale=0.6]{Example}
%\caption{Spectrum of A}
%\end{figure}
\end{example}\vspace{.2cm}


\begin{thebibliography}{10}
\expandafter\ifx\csname url\endcsname\relax
  \def\url#1{\texttt{#1}}\fi
\expandafter\ifx\csname urlprefix\endcsname\relax\def\urlprefix{URL }\fi
\expandafter\ifx\csname href\endcsname\relax
  \def\href#1#2{#2} \def\path#1{#1}\fi

%21
\bibitem{AE15}
A.~M. Akhmedov, S.~R. El-Shabrawy, Spectra and fine spectra of lower triangular
  double-band matrices as operators on $\ell_p (1 \leq p< \infty)$, Math. Slovaca 65~(5) (2015) 1137--1152.

%15
\bibitem{AE12}
A.~M. Akhmedov, S.~R. El-Shabrawy, Some {S}pectral {P}roperties of the
  {G}eneralized {D}ifference {O}perator ${\Delta}_v$, Eur. J. Pure Appl. Math. 5~(1) (2012) 59--74.

%13
\bibitem{AE2012}
A.~M. Akhmedov, S.~R. El-Shabrawy, Notes on the spectrum of lower triangular
  double-band matrices, Thai J. Math. 10~(2) (2012) 415--421.

%14
\bibitem{AE2011a}
A.~M. Akhmedov, S.~R. El-Shabrawy, On the fine spectrum of the operator
  ${\Delta}_v$ over the sequence spaces $c$ and $\ell_p, (1< p< \infty)$, Appl.
  Math. Inf. Sci. 5~(3) (2011) 635--654.

%17
\bibitem{AE2011}
A.~M. Akhmedov, S.~R. El-Shabrawy, On the fine spectrum of the operator
  ${\Delta}_{a, b}$ over the sequence space $c$, Comput. Math. Appl. 61~(10) (2011) 2994--3002.

%16
\bibitem{AE2010}
A.~M. Akhmedov, S.~R. El-Shabrawy, On the spectrum of the generalized
  difference operator ${\Delta}_{a, b}$ over the sequence space $c_0$, Baku
  Univ. News J, Phys. Math. Sci. Ser 4 (2010) 12--21.

  %3
\bibitem{AAFB}
A.~M. Akhmedov, F.~Ba{\c{s}}ar, The fine spectra of the difference operator
  ${\Delta}$ over the sequence space $bv_p, \ (1\leq p< \infty)$,Acta Math. Sin. (Engl. Ser.) 23~(10) (2007) 1757--1768.

 %25
\bibitem{AAFB1}
A.~M. Akhmedov, F.~Ba{\c{s}}ar, On the fine spectra of the difference operator
  ${\Delta}$ over the sequence space $l_p (1 \leq p < \infty)$, Demonstr. Math. 39~(3) (2006) 585--595.


 %4
\bibitem{brs_c0_c}
B.~Altay, F.~Ba{\c{s}}ar, On the fine spectrum of the generalized difference
  operator ${B}(r, s)$ over the sequence spaces $c_0$ and $c$, Int. J. Math.
  Math. Sci. 2005~(18) (2005) 3005--3013.

%27
%\bibitem{bib1}
%B.~Altay, F.~Ba\c{s}ar, On the fine spectrum of the generalized difference
%  operator ${B}(r,s)$ over the sequence spaces $c_0$ and $c$, Int. J. Math. Math. Sci. 18 (2005) 3005--3013.


 %26
\bibitem{Karakus}
B.~Altay, M.~Karaku, On the spectrum and the fine spectrum of the {Z}weier
  matrix as an operator on some sequence spaces, Thai J. Math.
  3~(2) (2005) 153--162.

%1
\bibitem{Basar04}
B.~Altay, F.~Ba\c{s}ar, On the fine spectrum of the difference operator
  ${\Delta}$ on $c_0$ and $c$, Inform. Sci. 168~(1-4) (2004) 217--224.

%10
\bibitem{altun}
M.~Altun, On the fine spectra of triangular toeplitz operators, Appl. Math.
  Comput. 217~(20) (2011) 8044--8051.

%28
\bibitem{bib5}
H.~Bilgi\c{c}, H.~Furkan, On the fine spectrum of the generalized difference
  operator ${B}(r,s)$ over the sequence spaces $\ell_p$ and $bv_p$, $1 < p <
  \infty$, Nonlinear Anal. 68~(3) (2008)
  499--506.

 %6
\bibitem{brs_lp_bvp}
H.~Bilgi{\c{c}}, H.~Furkan, On the fine spectrum of the generalized difference
  operator ${B}(r, s)$ over the sequence spaces $\ell_p$ and $bv_p$,$(1< p<
  \infty)$, Nonlinear Anal. 68~(3) (2008) 499--506.

 %8
\bibitem{brst_l1_bv}
H.~Bilgi{\c{c}}, H.~Furkan, On the fine spectrum of the operator ${B}(r, s, t)$
  over the sequence spaces $\ell_1$ and $bv$, Mathematical and computer
  modelling 45~(7) (2007) 883--891.

\bibitem{das}
R.~Das, On the spectrum and fine spectrum of the upper triangular matrix $U\left(r_1, r_2; s_1, s_2 \right)$ over the sequence space $c_0$, Afr. Mat. 28~(5-6) (2017) 841--849.

 %20
\bibitem{EL14}
S.~R. El-Shabrawy, Spectra and fine spectra of certain lower triangular
  double-band matrices as operators on $c_0$, J. Inequal. Appl. 2014~(1) (2014) 241.

 %18
\bibitem{EL2012}
S.~R. El-Shabrawy, On the fine spectrum of the generalized difference operator
  ${\Delta}_{a, b}$ over the sequence space $\ell_p, (1< p< \infty)$, Appl.
  Math. Inf. Sci. 6 (2012) 111--118.

  %9
\bibitem{FURK2010}
H.~Furkan, H.~Bilgi\c{c}, F.~Ba\c{s}ar, On the fine spectrum of the operator${
  B} (r, s, t)$ over the sequence spaces $ \ell_p$ and $bv_p,(1< p<\infty)$,
  Comput. Math. Appl. 60~(7) (2010) 2141--2152.

 %7
\bibitem{brst_c0_c}
H.~Furkan, H.~Bilgi{\c{c}}, B.~Altay, On the fine spectrum of the operator
  ${B}(r, s, t)$ over $c_0$ and $c$, Comput. Math. Appl. 53~(6) (2007)
  989--998.


%5
\bibitem{brs_l1_bv}
H.~Furkan, H.~Bilgi\c{c}, K.~Kayaduman, On the fine spectrum of the generalized difference operator
  ${B}(r, s)$ over the sequence space $\ell_1$ and $bv$, Hokkaido Math. J.
  35~(4) (2006) 893--904.

  %22
\bibitem{Q}
I.~Gohberg, S.~Goldberg, M.~A. Kaashoek, Classes of linear operators,
  Birkhäuser, 1990.


 %2
\bibitem{Kayaduman}
K.~Kayaduman, H.~Furkan, The fine spectra of the difference operator ${\Delta}$
  over the sequence spaces $l_1$ and $bv$, Int. Math. Forum 1~(21-24) (2006) 1153--1160.

 %24
\bibitem{maddox}
I.~J. Maddox, Elements of functional analysis, CUP Archive, 1988.


%19
\bibitem{SK2}
P.~D. Srivastava, S.~Kumar, Fine spectrum of the generalized difference
  operator ${\Delta}_{uv}$ over the sequence space $\ell_1$, Appl. Math. Comput. 218~(11) (2012) 6407--6414.


%12
\bibitem{SK}
P.~D. Srivastava, S.~Kumar, Fine spectrum of the generalized difference
  operator ${\Delta}_v$ on sequence space $\ell_1$, Thai J. Math.
  8~(2) (2010) 221--233.


%11
\bibitem{SK1}
P.~D. Srivastava, S.~Kumar, On the fine spectrum of the generalized difference
  operator ${\Delta}_v$ over the sequence space $c_0$, Communications in
  mathematical analysis 6~(1) (2009) 8--21.

%23
\bibitem{wilansky}
A.~Wilansky, Functional analysis, Blaisdell Publishing Company, New York, 1964.






%\bibitem{bib8}
%H.~Furkan, H.~Bilgi\c{c}, K.~Kayaduman, On the fine spectrum of the generalized
%  difference operator ${B}(r,s)$ over the sequence spaces $l_1$ and $bv$,
%  Hokkaido mathematical journal 35~(4) (2006) 893--904.


\end{thebibliography}
\end{document}